\documentclass{article}
\usepackage{graphicx} 

\usepackage{amssymb}
\usepackage{answers}
\usepackage{setspace}
\usepackage{graphicx}
\usepackage{enumitem}
\usepackage{multicol}
\usepackage{mathrsfs}
\usepackage{hyperref}
\usepackage[margin=1in]{geometry} 
\usepackage{amsmath,amsthm,amssymb}
\usepackage{tikz}
\usetikzlibrary {arrows.meta,bending,positioning, fit}
\usepackage{xcolor}
\usepackage{adjustbox}
\usepackage{mathtools}

 \newcommand{\bbM}{\mathbb{M}}

\newcommand{\calI}{\mathcal{I}}
\newcommand{\calJ}{\mathcal{J}}
\newcommand{\calK}{\mathcal{K}}

 \newcommand{\dom}{\text{dom }}

\newcommand{\tp}{\text{tp}}

\newcommand{\qftp}{\text{qftp}}

\newcommand{\TPi}{\text{TP}_1}
\newcommand{\TPii}{\text{TP}_2}
\newcommand{\Age}{\text{Age}}

 \newcommand{\vphi}{\varphi}

\newtheorem{theorem}{Theorem}

\newtheorem{corollary}{Corollary}
\newtheorem{proposition}{Proposition}
\newtheorem{example}{Example}

\theoremstyle{definition}
\newtheorem{definition}{Definition}

\theoremstyle{definition}
\newtheorem{remark}{Remark}

\theoremstyle{definition}
\newtheorem{notation}{Notation}

\theoremstyle{definition}

\title{Results on Colored Tree Properties}
\author{Gabriel Day}
\date{June 2025}

\begin{document}

\maketitle

\begin{abstract}
    In this paper, we introduce novel variations on several well-known model-theoretic tree properties, and prove several equivalences to known properties. Motivated by the study of generalized indiscernibles, we introduce the notion of the $\calI$-tree property ($\calI$-TP), for an arbitrary Ramsey index structure $\calI$. We focus attention on the colored linear order index structure \textbf{c}, showing that \textbf{c}-TP is equivalent to instability. After introducing \textbf{c}-$\TPi$ and \textbf{c}-$\TPii$, we prove that \textbf{c}-$\TPi$ is equivalent to $\TPi$, and that \textbf{c}-$\TPii$ is equivalent to IP. We see that these three tree properties give a dichotomy theorem, just as with TP, $\TPi$, and $\TPii$. Along the way, we observe that appropriately generalized tree index structures $\calI^{<\omega}$ are Ramsey, allowing for the use of generalized tree indiscernibles.
    
\end{abstract}

\section{Introduction}

At least since Shelah's pioneering work in \cite{shelah_classification_1990}, model theorists have taken keen interest in certain combinatorial ``patterning properties" defined for a partitioned formula $\vphi(x,y)$. These properties typically require that some infinite set of tuples which fit in the $y$ variable exists, and that certain collections of $\vphi$-formulas obtained by substituting elements from this set are consistent or inconsistent. Typical examples include the Order Property, Tree Property, and Independence Property. 

Traditionally, interest has concentrated on classes of theories which are characterized by the negations of certain nice properties, such as the stable, simple, and NIP theories. The central results of this paper may be understood, in a way, as giving further insight into those properties of traditional interest---new characterizations of stability and NIP are included, and the \textit{absence} of a new characterization of simplicity, by similar methods, is an interesting result.

However, these results may also be understood as coming from an attempt to understand patterning properties in general, by varying the index structure and studying the properties which emerge. The systematic study of patterning properties was inaugurated by Shelah's notion of \textit{straight definition} in \cite{shelah_what_2000}, and much clarified and elaborated in \cite{bailetti_walk_2024}. The notion of a \textit{exhibiting $n$-patterns} there offers a way of varying the index structure and producing distinct properties defined for a partitioned formula $\vphi(x,y)$. Making use of that methodology, the central definitions of this paper are variations of tree properties defined for the index structure $\omega^{<\omega}$, but with the ``horizontal" index structure $\omega$ replaced with an arbitrary Ramsey index structure $\calI$. Each of these corresponds to a new collection of $n$-patterns to exhibit, in Bailetti's parlance.

In their work on arbitrary patterning properties, both Shelah and Bailetti make a distinction between those properties where $\vphi$ appears only positively in the definition, and others in which both $\vphi$ and $\neg\vphi$ are needed. The former are called \textit{positively definable} properties, and it is interesting to note which properties are positively definable. In their traditional definitions, OP, IP, and SOP, for instance, involve $\neg\vphi$, while TP, $\TPi$, and $\TPii$ are positively definable. Two of the central results of this paper, Theorems \ref{c-TP unstable theorem} and \ref{IP=cTP2 theorem}, give positive characterizations of OP and IP, respectively. More broadly, this paper provides a new class of patterning properties, all of which are positive. This provides model theorists with a new tool for investigating such properties.

An essential piece of motivation for the definitions offered here is the study of generalized indiscernibles. While one can make sense of the $\calI$-tree property for an arbitrary structure $\calI$, we are centrally interested in those index structures which, like the linear order for OP and the tree for TP, allow one to find appropriately indiscernible instances. In Section \ref{indiscernibles section}, we retread the basics of generalized indiscernibles, with an eye to the index structure of primary interest to this paper: the colored linear order. In Section \ref{tree indiscernibles section}, we discuss new notions of tree indiscernibility, which will provide the necessary homogeneous index structures needed to enable the central proofs of the paper.

In broad strokes, this paper endeavors to further two goals in the study of generalized indiscernibles. First, it seeks to generalize notions, such as \textit{dividing}, which are defined in terms of order indiscernibles, to the setting of generalized indiscernibles. Second, it employs these new notions to characterize properties of theories, as discussed above. Previous work has been done toward both goals. In their paper \cite{malliaris_shelah_2021_separation_theorem}, Malliaris and Shelah introduce \textit{shearing}, a notion similar to dividing but for generalized index structures. They used shearing to prove new results in the study of the $\trianglelefteq^*$ ordering of theories, demonstrating its use for classification theory. In part, the study of generalized dividing in this paper was an attempt to see what could be done with a simplified version of shearing; substantial differences emerged, as noted below. The study of \textit{collapse of indiscernibles} results in \cite{guingona2017characterizing} gives another example of how generalized indiscernibles can be used to characterize properties of theories.

As the title suggests, this paper lavishes attention on the colored linear order index structure, and the tree properties which make use of it. Why study this somewhat-obscure index structure? The traditional tree properties TP, $\TPi$, and $\TPii$ correspond to the generalized tree properties where the the structure $\calI$ is taken to just be a linear order, with no further structure. The colored linear order appears to be a minor variation on the linear order, making for an easy version of the new tree properties to study. But the fact that this variation yielded results interestingly different from the linear order case came to highlight a particular feature of the colored linear order which sets it apart from all of the classical index structures: it contains multiple quantifier-free 1-types, one for each different color. In fact, the new characterization of stability in terms of \textbf{c}-TP may actually be made in terms of any $\calI$-TP, where $\calI$ is a Ramsey index structure with at least two non-algebraic quantifier-free 1-types (Corollary \ref{generalized stability characterization}). 

This fact gives us a general piece of information about different tree property variations. While the NTP theories are a superset of the stable theories---carving out the simple theories, a fertile ground for model-theoretic inquiry, the negation of $\calI$-TP for any $\calI$ which has multiple quantifier-free 1-types simply re-invents the stability dividing line. This shows that part of what makes the linear order such a fruitful index structure is that all of its elements possess a single quantifier-free 1-type.

\subsection{Conventions and notations}
We adopt  many standard model-theoretic conventions for this paper. We take a complete, consistent theory $T$ with infinite models in the background, $M$ an arbitrary model of $T$, and $\bbM$ a monster model.

When not specified, we use letters like $x$ and $y$ for finite tuples of variables, $a$ and $b$ for tuples of arbitrary finite length from $\bbM$. Abusing notation, we may write $a\in M$ even when $a$ is a tuple of length greater than 1. When considering an indexed collection of tuples $(a_t)_{t\in \calI}$, for some index structure $\calI$, and  $\bar{s}=(s_0,s_1,\dotsm, s_{n-1})$ is an $n$-tuple of indices, we abbreviate by $a_{\bar{s}}$ the tuple $(a_{s_0},a_{s_1},\dotsm, a_{s_{n-1}})$. We use letters $p,q$, and $r$ to denote types, and the operator $\qftp_\calI(i)$ to denote the complete quantifier-free type of a tuple $i\in\calI$. 

For a structure $\calI$ (in particular in the context where $\calI$ is an index structure), we mean by the \textit{age} of $\calI$, $\Age(\calI)$, the set of all finitely generated substructures of $\calI$, up to isomorphism.

We readily use standard set-theoretic notation for trees: in particular, we use $\omega^{<\omega}$ for the set of functions from finite initial segments of $\omega$ to $\omega$, giving an infinitely branching tree of height $\omega$. Similarly, we use $\calI^{<\omega}$ for the set of functions from initial segments of $\omega$ to $\calI$. Each of these trees is partially ordered by $\trianglelefteq$. These trees are rooted with a node $\emptyset$.

\section{Generalized Indiscernibility}\label{indiscernibles section}

In model theory, it is often useful to study collections of elements (or tuples) from a structure which are particularly uniform---which ``look similar" from the perspective of that structure. One notion of uniformity is that of an \textit{indiscernible sequence}. Indiscernible sequences are central for the definition of many model-theoretic terms---in particular, \textit{dividing}, \textit{forking}, and the entire study of independence notions. This naturally suggests that other forms of indiscernibility may be a rich ground for research. \textit{Generalized indiscernibles} are found in some form in \cite{shelah_classification_1990}, and have been employed in numerous contexts since then. I primarily follow the presentation and treatment from \cite{scow_characterization_2011}. 

\begin{definition}{}\label{generalized indiscernibles}
Suppose $\calI$ is an $L'$-structure and $M$ is an $L$-structure. Let $A=(a_i)_{i\in \calI}$ be a collection of tuples from $M$. Then we say that the $A$ is a \textit{generalized indiscernible (set)}, and in particular, an $\calI$\textit{-indexed indiscernible (set)}\footnote{Whenever indiscernibility is discussed in this paper, one could add ``over a set of parameters $C$", as usual. Unless otherwise stated (e.g., in the definition of dividing) we work over the empty set of parameters, for convenience.} if, for all $n\geq 1$ and $i_0,i_1,\dotsm, i_{n-1},j_0,j_1,\dotsm, j_{n-1}\in \calI$, if $\qftp_\calI(i_0,\dotsm, i_{n-1})=\qftp_\calI(j_0,\dotsm, j_{n-1})$, then $$\tp_M(a_{i_0},\dotsm a_{i_{n-1}})=\tp_M(a_{j_0},\dotsm, a_{j_{n-1}}).$$
\end{definition}

Indiscernible sequences found so much utility in part because they are so plentiful. Applications of Ramsey's Theorem and compactness allow one to find indiscernible sequences that ``look like" any arbitrary sequence of tuples in the monster model of any theory. This property, \textit{modeling property}, carries over to all other notable kinds of indiscernibles. What it means for a collection of tuples indexed by some structure to ``look like" another is the following:

\begin{definition}
    Suppose $\calI$ is a structure expanding a linear order, $M,N$ are models of an $L$-theory, and $A=\{a_i\colon i\in \calI\}$ is a set of parameters in $M$. Then we say that a set of parameters $B=\{b_i\colon i\in \calI\}$ from $N$ is \textit{locally based} on $A$, (or that the $b_i$ are \textit{locally based} on the $a_i$), if for any finite set $\Delta$ of $L$-formulas, and any finite tuple $\bar{t}=t_1\dotsm t_n$ from $\calI$, there exists a tuple $\bar{s}=s_1\dotsm s_n$ in $\calI$ such that \[\qftp_\calI(\bar{t})=\qftp_\calI(\bar{s})\enspace \text{and}\enspace \tp_N^\Delta(\bar{b}_{\bar{t}})=\tp_M^\Delta(\bar{a}_{\bar{s}}).\]
\end{definition}

Thus we may define the modeling property.

\begin{definition}{}
For a structure $\calI$, we say that $\calI$-indexed indiscernibles have the \textit{modeling property} if, given any $\calI$-indexed set $(a_i\colon i\in\calI)$ in an arbitrary structure $M$, there exist $\calI$-indexed indiscernibles $(b_i\colon i\in\calI)$ in $N\succeq M$ locally based on the $a_i$. Equivalently, one can find such $b_i$ in the monster model of $M$'s theory.

\end{definition}

In \cite{scow_characterization_2011}, Scow proved that, under some natural assumptions, a structure $\calI$ has the modeling property precisely when its age is a \textit{Ramsey class}, i.e., when a structural version of Ramsey's theorem holds.\footnote{See \cite{Kim_Kim_Scow_2013} for details on the Ramsey property. The modeling property is what concerns us in this paper, so we avoid deeper discussion of the Ramsey property, but are happy to borrow this name for ``good" index structures, in keeping with convention.} This result extends the classical proof, using Ramsey's Theorem, that infinite linear orders have the modeling property to the generalized indiscernible setting. We state the theorem in a form with the technical assumptions dropped, as was proved possible in \cite{meir_popadopoulos_2022practical}.

\begin{theorem}\label{Scow's theorem}
Suppose $\calI$ is a  
locally finite structure in a language $L$ including a relation $<$ linearly ordering its domain. Then  $\Age(\calI)$ is a Ramsey class if and only if $\calI$-indexed indiscernibles have the modeling property.    
\end{theorem}

Henceforth, index structures with Ramsey ages, or equivalently, those with the modeling property, will be called \textit{Ramsey index structures}. Many natural examples of Ramsey index structures have been discovered, and their correspondent types of indiscernibles have found deep and widespread use. The ordered random graph in the language $\{<, R\}$ is Ramsey, and it was employed in \cite{SCOW_2012_NIP} to prove a collapse of indiscernibles result for the Independence Property. Tree structures in various languages have been among the earliest and most fruitful Ramsey structures, appearing first in\cite{shelah_classification_1990}, receiving treatment in \cite{Kim_Kim_Scow_2013} and \cite{TAKEUCHI_Tsuboi_2012}, and finding numerous applications in the study of model-theoretic tree properties (such as $\TPi$ and $\TPii$) in \cite{Chernikov_Ramsey_2016} and \cite{Kaplan_Ramsey_2017}, for example. We will return to tree indiscernibles in Section \ref{tree indiscernibles section}.

As Ramsey theory is a subject of mathematical interest independent of model theory, there are many classes of structures known to be Ramsey, and techniques for producing more. With these in hand, Scow's Theorem \ref{Scow's theorem} provides a dictionary which furnishes model theorists with numerous ready-made Ramsey index structures. A particularly useful result from Ramsey theory, which will give us our main index structure for this paper, is due to Ne\v set\v ril and R\"odl \cite{NESETRIL_Rodl_1983}. It states that if $\tau$ is a relational language and $<$ a binary relation, then the class of all $\tau\cup\{<\}$-structures in which $<$ is interpreted as a linear order is a Ramsey class.

This brings us to the \textit{colored linear order}, a central index structure for the results which follow.

\begin{definition}{}\label{colored linear order}
Let $L_\kappa$  be the language $\{<, (c_\alpha)_{\alpha<\kappa}\}$, where $2\leq \kappa$ is a cardinal, each $c_\alpha$ is a unary predicate. Let $\calK_{color,\kappa}$be the class of finite structures in which $<$ is interpreted as a linear order, and the $c_\alpha$ are nonempty and partition the domain (i.e., they color it); we call each $c_\alpha$ a \textit{color}.

By the \textbf{colored linear order} index structure, denoted $\textbf{c}_\kappa$, we mean the Fra\"iss\'e limit of $\calK_{color,\kappa}$. When the value of $\kappa$ is irrelevant or clear from context, we shall simply write \textbf{c} in place of $\textbf{c}_\kappa$.
\end{definition}

This structure \textbf{c} is a Ramsey, by the result of Ne\v set\v ril and R\"odl above. The reduct of \textbf{c} to $<$ is a dense linear order. All of the colors are dense in this structure.\footnote{To my knowledge, the only other place that infinite colored linear orders have been used as an index structure is in \cite{MALLIARIS_Shelah_2019}, in particular for the proof that the theory of the random graph is not maximal in the $\trianglelefteq^*_1$ order (Theorem 5.5). Malliaris and Shelah sometimes utilize colored linear orders which are ``separated", i.e., where every element of the linear order has  a unique color. We are interested in the Fr\"iass\'e limit, in which each color is dense.} Even though \textbf{c}, properly speaking, has this order type, compactness allows us to be flexible with what colored order we consider, as long as we maintain the correct age.

\begin{remark}\label{changing order type}{}
Suppose that $\calI,\calI'$ are infinite structures in the same language such that $\Age(\calI)=\Age(\calI')$. Given a set of tuples $(a_i)_{i\in\calI}$, there exists a set of tuples $(a'_i)_{i\in\calI'}$ locally based on the $a_i$. Likewise, given a set of tuples $(b_i)_{i\in\calI'}$, there exists a set of tuples $(b'_i)_{i\in\calI}$ locally based on the $a_i$.
    
\end{remark}

This follows by an elementary application of the compactness theorem. As we will see, this fact affords us flexibility to take \textbf{c} to be of a convenient order type, when we wish. Abusing notation, we may just specify that \textbf{c} \textit{is} such-and-such a structure, as long as the age is the same. Often, we will take  the two-colored linear order to have the order type of $\omega$, with even elements being one color (call it ``red") and odd elements another (call it ``green").

\begin{notation}
    Since we will often be concerned with quantifier-free types in the language of colored linear orders, we introduce the following conventional notation. Suppose that $c_1,\dotsm, c_k$  are color predicates. Then we may denote by $c_1<\dotsm <c_k$  the quantifier-free type of $k$-many elements in increasing order, such that the $i$th element has color $c_i$. When $R$ and $G$ are used as color predicates, we may , for instance, refer to the quantifier-free type $R<G$.
\end{notation}

\section{New tree indiscernibilities}\label{tree indiscernibles section}
While \textbf{c} is, in a sense, the only particular Ramsey index model we are concerned with in this paper, there is another class of index models which are essential. These are natural generalizations of the tree index structures, and give generalizations of the tree indiscernibilities, which have received considerable prior attention. I follow the terminology for these structures established in \cite{Kim_Kim_Scow_2013}.

Three main languages for trees have been used to treat them as index structures. All of them contain the tree partial order $\trianglelefteq$, as well as the lexicographic order $<_{lex}$ and the meet function $\land$. The first of these languages, $L_0$, consists of only these symbols. A second language, $L_{str}$, adds a  preorder, $<_{len}$, which compares the lengths of nodes. A final language, $L_s$, omits the length partial order but adds an infinite collection of unary predicates, $(P_\alpha)_{\alpha<\kappa}$, where $\kappa$ is the height of the tree. The $\alpha$th of these predicates is taken to hold of precisely the nodes of length $\alpha$. In this paper, we shall only be concerned with $L_0$ and $L_s$.

When regarded as $L_0$- and $L_s$-structures, the trees $\lambda^{<\kappa}$, where $\lambda$ and $\kappa$ are infinite cardinals, are all Ramsey index structures. They are thus well-suited to serve as index structures for indiscernibles. The interest in  tree indiscernibles lies centrally in the fact that tree properties, like TP and $\TPi$, require that there be a tree-indexed collection of parameters $(a_\eta)_{\eta\in \omega^{<\omega}}$ satisfying various properties (See Definitions \ref{tree property} and \ref{TP1 definition}), and it is often helpful, when manipulating such a tree of parameters, that it have the uniformity guaranteed by indiscernibility. However, witnesses to TP and $\TPi$ cannot be taken to have the same ``degree" of indiscernibility, since they demand different things of the tree of parameters: TP can, in general, only have the weaker $L_s$-indiscernibility, while a witness to $\TPi$ can be taken as $L_0$-indiscernible.\footnote{In general, if $\calJ$ is a reduct of $\calI$, we have that $\calJ$ indiscernibility implies $\calI$ indiscernibility. This is immediate from Definition \ref{generalized indiscernibles}, since the reduct language has fewer quantifier-free types, and so more tuples are required to have the same types in order to attain indiscernibility.}

Some of the central results of this paper concern generalizations of TP and $\TPi$ which demand the existence of a collection of parameters indexed by $\calI^{<\kappa}$, where $\calI$ is a Ramsey index structure. As in the classical context, indiscernible witnesses to these properties are often useful. But $\calI^{\kappa}$ calls for a richer language than the tree languages above, and thus of new indiscernibility results. Fortunately, both are easily adapted.

\begin{definition}{}\label{I trees}
    Suppose $\calI$ is a Ramsey index structure in the language $L_\calI$. Let $L_\calI'$ be the $L_\calI$ but with all $n$-ary function symbols replaced with $n+1$-ary relation symbols. Let $L_{0, \calI}=\{\trianglelefteq,<_{lex},\land\}\cup L_\calI'$. Let $L_{s,\kappa,\calI}=\{\trianglelefteq,<_{lex},\land,(P_\alpha)_{\alpha<\kappa}\}\cup L_\calI'$, for $\kappa$ an infinite cardinal (corresponding to the tree's height).
    
\end{definition}

On the set $\calI^{<\kappa}$, these languages are interpreted as follows: The tree partial order $\trianglelefteq$ is interpreted as usual, as is the lexicographic order, which respects the total ordering of $\calI$. In $P_{\alpha}$ is taken to hold of all $\nu\in \calI^{\alpha}$, i.e., the $\alpha$th level of the tree. $L_\calI'$ is interpreted such that given any $\eta\in\calI^{<\kappa}$, the set of nodes immediately following $\eta$ is isomorphic to $\calI$ in the obvious $L_\calI'$ translation in which function symbols are replaced by their graphs. No relations are taken to hold between different copies of $\calI$. The replacement of function symbols with relations is to avoid the need to define functions taking values from different copies of $\calI$.

The traditional notions of tree indiscernibility may be seen as special cases of this definition. This may be done by regarding $\omega$ as $(\omega,<)$, an infinite linear order and Ramsey index structure. The linear order introduced is redundant with the lexicographic order restricted to each copy of $\omega$ in the tree.

The fact that $(\omega, <)$ is Ramsey (i.e., that Ramsey's Theorem holds) figures into the proof that $L_{s,\kappa}$-trees have the modeling property, in particular in Claim 4.9 of \cite{Kim_Kim_Scow_2013}. That proof goes through identically to show that $L_{s,\kappa,\calI}$-trees have the modeling property for $\calI$ Ramsey, simply using the fact that $\calI$ has the modeling property when one comes to Claim 4.9. Thus we have the following:

\begin{proposition}{}\label{L_s,I trees have modeling property}
    Suppose $\calI$ is a Ramsey index structure. When $\calI^{<\omega}$ is given an $L_{s,\kappa,\calI}$-structure, it has the modeling property, and hence $\text{Age}(\calI^{<\omega})$ is a Ramsey class.

\end{proposition}


Extending to the stronger notion of indiscernibility, that which uses $L_{0,\calI}$, likewise only makes use of the fact that $(\omega,<)$ is Ramsey. This may be seen in \cite{TAKEUCHI_Tsuboi_2012}, where the authors bootstrap the fact that $L_{s}$ trees have the modeling property to find that $L_{str}$- and finally $L_0$- trees likewise have it (Theorems 15 and 16, respectively).\footnote{Some differences in terminology slightly obfuscate matters---for instance, they call $str$-indiscernibility 1-indiscernibility, and the several ``subtree properties" are easily seen to correspond to trees in the respective languages having the modeling property.} Thus, the correspondent generalized tree indiscernibilities also go through.

\begin{proposition}{}\label{L_0,I trees have modeling property}
    Suppose $\calI$ is a Ramsey index structure. When $\calI^{<\omega}$ is given an $L_{0,\calI}$-structure, it has the modeling property, and hence $\text{Age}(\calI^{<\omega})$ is a Ramsey class.

\end{proposition}

The final notion of indiscernibility needed in this paper is a correspondent of strong array indiscernibility, which will be useful in discussion of generalized $\TPii$ below. While $\TPii$ is defined in terms of arrays $(a_{i,j})_{i<\kappa,j<\lambda}$ (see Definition \ref{TP2 definition}), we will use as index structure arrays in which the rows are copies of some Ramsey index structure $\calI$, namely $(a_{i,j})_{i<\kappa,j\in\calI}$.

\begin{definition}
We call an array $A=(a_{i,j})_{i<\kappa,j<\lambda}$ \textit{mutually indiscernible} if, for all $i<\kappa$, the sequence $(a_{i,j})_{j<\lambda}$ is indiscernible over $(a_{k,j})_{k\neq i, j<\lambda}$. $A$ is a \textit{strongly indiscernible array} if it is a mutually indiscernible array such that the sequence of rows, $(\{a_{i,j}\colon j<\lambda\})_{i<\kappa}$, is an indiscernible sequence. 
\end{definition}

By Lemmas 1.2 and 1.3 of \cite{chernikov2014theories}, given any array indexed by $\kappa\times\lambda$ with the structure of a convexly ordered equivalence relation (where the rows are equivalent), one may find a strongly indiscernible array locally based on it. As with the tree results, only the fact that rows have Ramsey index structures is needed for the proof. Thus we may generalize. 

\begin{definition}
 We call an array $A=(a_{i,j})_{i<\kappa,j\in\calI}$ \textit{mutually $\calI$-indiscernible} if, for all $i<\kappa$, the sequence $(a_{i,j})_{j\in\calI}$ is indiscernible over $(a_{k,j})_{k\neq i, j\in\calI}$. $A$ is a \textit{strongly $\calI$-indiscernible} array if it is a mutually $\calI$-indiscernible array such that the sequence of rows, $(\{a_{i,j}\colon j\in\calI\})_{i<\kappa}$, is an indiscernible sequence. 
\end{definition}

We likewise give $\kappa\times \calI$ the structure of a convexly ordered equivalence relation, where the copies of $\calI$ are equivalence classes. As in Definition \ref{I trees}, we replace function symbols in $L_\calI$ with relations which, for each $\alpha<\kappa$, give the graphs of the corresponding function symbols on $\{\alpha\}\times\calI$. This makes $\{\alpha\}\times\calI$ isomorphic to $\calI$.

\begin{proposition}\label{array indiscernibility}
    Suppose $\calI$ is a Ramsey index structure. Given an array $(a_{i,j})_{i<\kappa,j\in\calI}$, we may find a strongly $\calI$-indiscernible array locally based on it.
\end{proposition}

\section{Generalized dividing and $q$-inconsistency}

Generalized dividing is, naturally, a generalization of the classical notion of dividing, replacing the linear order with different kinds of Ramsey index structures.

\begin{definition}\label{dividing defn}  
Suppose $\vphi(x,y)$ is an $L$-formula, $M$ an $L$-structure, $a\in M$ is a tuple of the same length as $y$, and $C\subseteq M$. Then $\vphi(x,a)$ \textit{divides over $C$} if there is a $C$-indiscernible sequence of tuples from $M$, $(a_i)_{i\in\omega}$, with $a_0=a$ such that the set of formulas $\{\vphi(x,a_i)\colon i\in\omega\}$ is inconsistent.
\end{definition}

Generalized dividing is defined in the most straightforward way possible: simply replace the indiscernible sequence in the definition of dividing with a generalized indiscernible.\footnote{Generalized dividing can be seen as a simplified version of the notion of \textit{shearing}, defined in \cite{malliaris_shelah_2021_separation_theorem}. Shearing likewise generalizes dividing to address indiscernibles with richer index structures. However, a number of differences exist between the notions. Most importantly for the contents of this paper, generalized dividing allows for inconsistency to be given by indices with different quantifier-free types, such as the different colors of \textbf{c}. Shearing, by contrast, requires that inconsistency occur among elements indexed by tuples with the same quantifier-free type.}

\begin{definition}
   Suppose $\vphi(x,y)$ is an $L$-formula, $\calI$ is an index structure, $M$ an $L$-structure, $a\in M$ is a tuple of the same length as $y$, and $C\subseteq M$. Then $\vphi(x,a)$ $\calI$-\textit{divides over} $C$ (and for arbitrary $\calI$, \textit{generalized divides over $C$}) if there is a $\calI$-indexed $C$-indiscernible set of tuples from $M$, $(a_i)_{i\in\calI}$, with $a=a_j$ for some $j\in\calI$, such that the set of formulas $\{\vphi(x,a_i)\colon i\in\calI\}$ is inconsistent. 
\end{definition}

When a formula divides, the set of formulas $\{\vphi(x,a_i)\colon i\in\omega\}$ is said to be $k$-inconsistent, for a natural number $k$. The notion of $k$-inconsistency falls out of the assumption of inconsistency. By compactness, some finitely many (say $k$-many) formulas $\{\vphi(x, a_{i_0}),\dotsm, \vphi(x,a_{i_{k-1}})\}$ are inconsistent, where $i_0<\dotsm <i_{k-1}$. Since sequences witnessing dividing are indiscernible, it follows that for any $j_0<\dotsm<j_{k-1}$, $\{\vphi(x, a_{j_0}),\dotsm, \vphi(x,a_{j_{k-1}})\}$ is also inconsistent---hence we get that any $k$ formulas from the sequence are inconsistent, i.e., $k$-inconsistency. 

To address index structures where not all $k$-tuples have the same quantifier-free type, $k$-inconsistency must be generalized. To see how, consider an index structure $\calI$ which expands a linear order. If the set of formulas $\{\vphi(x,a_i)\colon i\in\calI\}$ is inconsistent, a finite piece of it, $\{\vphi(x, a_{i_0}),\dotsm, \vphi(x,a_{i_{k-1}})\}$, must be inconsistent, by compactness. The indexing sequence may, by the modeling property, be taken to be (generalized) indiscernible, in which case any tuple $(j_0,\dotsm, j_{k-1})$ with the same quantifier-free type as $(i_0,\dotsm, i_{k-1})$ must be such that $\{\vphi(x, a_{j_0}),\dotsm, \vphi(x,a_{j_{k-1}})\}$ is inconsistent. Thus, there must be some complete quantifier-free type $q$ such that any tuples satisfying it give inconsistent formulas. This motivates the following definition: 

\begin{definition}{($q$-inconsistent)} Given a formula $\vphi(x,y)$ and a set of tuples $(a_i)_{i\in\calI}$, where each $a_i$ is the appropriate length for the $y$ variable, and a complete quantifier-free $k$-type $q$ in the language of $\calI$, we say that the set of formulas $\{\vphi(x,a_i)\colon i\in\calI\}$ is $q$\textit{-inconsistent} if, whenever $i_0,\dotsm, i_{n-1}\in \calI$ are such that $(i_0,\dotsm, i_{k-1})\models q$, then the set of formulas $$\{\vphi(x, a_{i_0}),\dotsm, \vphi(x,a_{i_{k-1}})\}$$
is inconsistent.
\end{definition}

It is crucial to remember that, for some index structures $\calI$, when there is $q$-inconsistency, there may be an infinite subsequence of $\calI$ indexing a set of consistent formulas. This will often be the case when considering colored linear orders, where $q$ is the type of a red index less than a green index, say, and all of the red-indexed elements, taken together, are consistent, yet any red and any green are inconsistent. This occurs, for instance, in the proof that the random graph has the colored tree property, Proposition \ref{random graph C-TP} below.

Many potential avenues of research spring from the notion of generalized dividing, most of which involve replacing the classical notion with specific generalized index structures. In this vein, we first define a generalization of local character, which in the classical context, is one way of defining simple theories. It is not, as it turns out, a good generalization of that notion to index structures such as \textbf{c}, as we will see below. However, it can be characterized, as usual, with a generalization of dividing chains.

\begin{definition}{}
$\calI$-dividing has \textit{local character} for a theory $T$ when every type does not $\calI$-divide over some subset of its domain of size at most $|T|+\kappa$, where $\kappa$ is the number of quantifier-free types in finitely many variables in $\calI$. 
\end{definition}

\begin{definition}{}
A $\calI$\textit{-dividing chain} for $\vphi(x,y)$ is a sequence $(a_i\colon i<\alpha)$ such that $\{\vphi(x, a_i)\colon i<\alpha\}$ is consistent, and $\vphi(x,a_i)$ $\calI$-divides over $a_{<i}$. If each instance of dividing in the chain is witnessed in a single quantifier-free type $q$ in the index models, we can call this $q$-inconsistent $\calI$-dividing.
\end{definition}

\begin{proposition}{}
$T$ has $\calI$-local character if and only if there is no formula $\vphi(x,y)$, sequence $(a_i\colon i<\omega)$ in the monster model of $T$, and quantifier-free type in the language of $\calI$ which form a $q$-inconsistent $\calI$-dividing chain. 
\begin{proof}
Let $\lambda=|T|+\kappa$.

For one direction, suppose there is a $q$-inconsistent $\calI$-dividing chain $(a_i\colon i<\omega)$. 

First, we note that this can be extended to a $q$-inconsistent $\calI$-dividing chain of length $\alpha$, for any ordinal $\alpha$, by a straightforward compactness argument.

 Since we can extend our dividing chains, we can assume we actually had a $\lambda^+$-length $q$-inconsistent $\calI$-dividing chain $(a_i)_{i<\lambda^+}$. Let $A=\{a_i\colon i<\lambda^+\}$. Then $p=\{\vphi(x,a_i)\colon i<\lambda^+\}$ is a type with domain $A$ (note that the consistency requirement for a dividing chain guarantees that $p$ is in fact a type). However, $p$ $\calI$-divides over any subset of its domain of size $\leq |T|+\kappa$, since the indices of any such subset must be bounded above by some $i<\lambda^+$, and hence $p$ $\calI$-divides over $a_{<i}$, since $\vphi(x,a_i)$ does. Thus we have a failure of local character for $T$.

Suppose local character fails for $T$, i.e. there is a type $p\in S(A)$ which $\calI$-divides over every subset of its domain of size $\leq \lambda$. If $\vphi(x,b)\in p$ $\calI$-divides over $A$, then the constant sequence where $b_i=b$ is an $\calI$-dividing chain for $\vphi(x,y)$.\footnote{Note that Proposition \ref{triviality} shows that there are cases where types $\calI$-divide over their domains.} So we may assume $p$ does not $\calI$-divide over its domain, and thus $|A|\geq \lambda^+$.

Begin by inductively constructing a sequence of formulas with parameters $(\vphi_i(x,b_i)\colon i<\lambda^+)$, each of which is in $p$, such that $\vphi_i(x,b_i)$ $\calI$-divides over $b_{<i}$. Supposing we have constructed the sequence up to stage $i$, we note that $b_{<i}$ is a subset of $A$ of size at most $|T|+\kappa$, and thus $p$ $\calI$-divides over it. Thus $p$ includes a formula $\psi(x, d)$ which $\calI$-divides over $b_{<i}$; simply let $\psi=\vphi_{i+1}$ and $d=b_{i+1}$. Because $|A|\geq \lambda^+$, we can construct a long enough sequence. 

Now, consider the pairs $(\vphi_i,q_i)$ consisting of a formula in the sequence, and the quantifier-free type in the index language witnessing the $\calI$-dividing of $\vphi_i(x,b_i)$ over $b_{<i}$. Since there are $|T|$-many such formulas, and $\kappa$ many such quantifier-free types, yet the list has length $\lambda^+$, we may extract a countable constant subsequence by the pigeonhole principle. This furnishes a formula $\vphi(x,y)$ and quantifier-free type $q$ such that the subsequence of the $(b_i)_{i<\lambda^+}$, which we may call $(a_i)_{i<\omega}$, is the desired $q$-inconsistent $\calI$-dividing chain of length $\omega$ for $\vphi$.
\end{proof}
\end{proposition}

As a seemingly minor alteration to the classical linear order index model, \textbf{c} was the first generalized index structure with which we investigated generalized local character. However, this notion turns out to be trivially strong for \textbf{c}, and thus fails to carve out any interesting classificatory property of theories. Its triviality may be seen by the following fact.

\begin{proposition}\label{triviality}
    Let $T_\infty$ be the theory of infinite sets, in the language with only $=$. Then $T_\infty$  does not have \textbf{c}-local character.

    \begin{proof}
        Consider the formula $x=y$. From $M\models T_\infty$, select two distinct elements, $a$ and $b$. Define $(a_i)_{i\in\omega}$ to be the constant sequence where $a_i=a$; this gives a \textbf{c}-dividing chain for $x=y$. Clearly the consistency condition is met. For $a_i$, the sequence $a,b,a,b,a,b,\dotsm$ gives a two-colored indiscernible sequence over $a_{<i}$, where odds and evens are given distinct colors. This gives inconsistent formulas for $x=y$. Hence we have an infinite dividing chain, and no \textbf{c}-local character.
    \end{proof}
\end{proposition}

The failure of \textbf{c}-local character to define an interesting analog of simplicity led naturally to the investigation of analogs of the tree property. Classically, having local character is equivalent to omitting the tree property. That would suggest that the colored analog of the tree property would also be trivial. But as we shall see, that is not the case.\footnote{Another direction of research which would be natural, having defined generalized dividing, is the investigation of associated notions of forking and independence, particularly for the colored linear order. We hope to pursue this avenue in future research.}

\section{The Colored Tree Property}

The original version of the tree property is defined as follows.

\begin{definition}\label{tree property}
    A formula $\vphi(x,y)$ has the \textit{tree property} if there exist  tuples $(a_\eta)_{\omega^{<\omega}}$ satisfying the following:

    \begin{enumerate}
        \item \textbf{Paths are consistent:}   For each path $\beta\in\omega^{\omega}$, the set of formulas $\{\vphi(x,a_{\beta|_n})\colon n<\omega\}$ is consistent.

        \item \textbf{Siblings are $k$-inconsistent:} There is a natural number $k$ such that for each $\eta\in\omega^{<\omega}$, any $k$ elements of the set $\{\vphi(x,a_{\eta^\smallfrown i})\colon i<\omega\}$ is inconsistent.
    \end{enumerate}
\end{definition}

At first glance, there are (at least) two natural ways to generalize this to trees involving a richer index structure $\calI$, and in particular, to \textbf{c}. Certainly, $k$-inconsistency of siblings must be replaced by $q$-inconsistency. Consequently, the successors of any node in the indexing tree must form a copy of $\calI$. For the consistency condition, one might either ask that all paths are consistent, or just the parts of paths indexed by things with the same quantifier-free types (e.g., same color). It turns out that the latter option is equivalent to the trivial notion of \textbf{c}-local character, so the former is more promising. We arrive at the following definition.

\begin{definition}{($\calI$-Tree Property)}
        For an index structure $\calI$, we say that a formula $\vphi(x,y)$ has the $\calI$-tree property ($\calI$-TP) when there exists a tree of parameters $(a_\eta)_{\eta\in \calI^{<\omega}}$ and a complete quantifier-free $n$-type $q$ in the language of $\calI$, such that:

        \begin{enumerate}
    \item \textbf{Paths are consistent:} For all $\beta\in\calI^{\omega}$, $\{\vphi(x, a_{\beta\mid _i})\colon i<\omega\}$ is consistent.

   \item \textbf{Siblings are $q$-inconsistent:} For every $\eta\in\calI^{<\omega}$, the set of formulas indexed by immediate successors of $\eta$, i.e., $\{\vphi(x,a_{\eta^\smallfrown i_k})\colon k<\omega\}$, are $q$-inconsistent.
   

As usual we say that a model and a theory, have $\calI$-TP and N$\calI$-TP.
\end{enumerate}
\end{definition}

By Proposition \ref{L_s,I trees have modeling property}, given an tree of parameters indexed by $\calI^{<\omega}$, we can find an $L_{s,\omega,\calI}$-insiscernible tree locally based on it. Local basedness in this language preserves witnessing $\calI$-TP, so we may always take $\calI$-TP to be witnessed by an $L_{s,\omega,\calI}$-indiscernible tree of parameters. 

We will see below in Theorem \ref{c-TP unstable theorem} that \textbf{c}-TP is not trivially strong, as \textbf{c}-local character is. At the same time, it is genuinely different from TP. This difference is demonstrated by the fact that the Rado random graph has \textbf{c}-TP. The random graph is a structure of interest for many reasons, but one of them  is that it is an unstable theory which is simple, i.e., it omits TP.

\begin{proposition}{}\label{random graph C-TP}
The random graph has $\textbf{c}_2$-TP.
\begin{proof}

Assume, for this proof, that $\textbf{c}_2$ has order type $\omega$, with reds even and greens odd; recall we may assume this index structure, by Remark \ref{changing order type}. The formula witnessing $\textbf{c}_2$-TP is $\vphi(x,yz)=xRy\land\neg xRz$. The tree of parameters witnessing that $\vphi$ has $\textbf{c}_2$-TP will consist of pairs of disconnected nodes. For each $i\in\omega$, fix a pair of distinct, disconnected nodes $c_i, d_i$; these should also be distinct from $c_j,d_j$ for all $j\neq i$. 

Now, define a tree of pairs $(b_\eta)_{\eta\in \textbf{c}_2^{<\omega}}$ as follows: Suppose the tree has been constructed up to level $i-1$, and $\eta$ has length $i-1$. If $n$ is even and red, define $b_{\eta^\smallfrown n}=c_id_i$. If $n$ is odd and green, define $b_{\eta^\smallfrown n}=d_ic_i$.

I claim that $(b_\eta)_{\eta\in \textbf{c}_2^{<\omega}}$ witnesses $\textbf{c}_2$-TP for $\vphi$. If $i$ is red and $j$ is green, $\{\vphi(x,b_{\eta^\smallfrown i}),\vphi(x,b_{\eta^\smallfrown j})\}$ asserts $xRc_i$ and $\neg xRc_i$, hence these are inconsistent. Paths are consistent by the random graph axioms, together with the fact that the $c_i$ and $d_i$ were chosen to be distinct at different levels of the tree. Thus the $\textbf{c}_2$-TP definition is satisfied for $\vphi$.

\end{proof}
\end{proposition}

That the random graph has $\textbf{c}_2$-TP actually implies that it has every $\textbf{c}_\kappa$-TP. 

\begin{proposition}{}\label{changing the number of colors}
If a formula has $\textbf{c}_m$-TP, then it has $\textbf{c}_n$-TP for any $n>m$, with inconsistency witnessed by the same quantifier-free type $q$. In particular, $\textbf{c}_2$-TP implies any $\textbf{c}_\kappa$-TP. The converse is not necessarily true.

\begin{proof}
    In order to produce a witness for $\textbf{c}_n$-TP, modify the tree witnessing $\textbf{c}_m$-TP by duplicating the tuples in each copy of $\textbf{c}_m$ indexed by the original $m$ colors and assigning them the new colors. Since the same tuples appear, we maintain the consistency of paths. Meanwhile, the new tuples are irrelevant to the $q$-inconsistency, so we have $\textbf{c}_n$-TP.
\end{proof}
\end{proposition}

The random graph is not just \textit{a} simple unstable theory, but \textit{the} typical structure with the Independence Property (IP) but not TP. Thus, it follows that all structures with IP have \textbf{c}-TP. The following definition of IP in terms of the \textit{alternating number} is not the most traditional, but will be useful in the proof of Theorem \ref{IP=cTP2 theorem} below. See Lemma 2.7 of \cite{Simon_2015} for a proof that it is equivalent to a more standard definition.

\begin{definition}
    A formula $\vphi(x,y)$ is said to have the \textit{Independence Property} (IP) if there is an indiscernible sequence $(a_i)_{i<\omega}$ such that the following set of formulas is consistent: $$\{\vphi(x,a_i)\colon i\enspace\text{is even}\}\cup\{\neg\vphi(x,a_i)\colon i\enspace\text{is odd}\}.$$
\end{definition}

\begin{corollary}{}
If $T$ has IP, then $T$ has $\textbf{c}$-TP.
\begin{proof}
   By Lemma 2.2 of \cite{Laskowski_Shelah_Karp_complexity_IP}, any theory with IP interprets a graph with the countable random graph as an induced subgraph. The example from the previous proposition may be reproduced in this induced subgraph, witnessing $\textbf{c}_2$-TP, and thus arbitrary \textbf{c}-TP.
\end{proof}
\end{corollary}

\begin{proposition}{}
 If $T$ has TP, then $T$ has $\textbf{c}$-TP.

\begin{proof}
    This is an extension of Proposition \ref{changing the number of colors} to the case of one color. Any witness to TP is a witness to $\textbf{c}$-TP; just color siblings arbitrarily. If $T$ has $k$-TP, the quantifier-free type $q$ witnessing inconsistency may be taken to be that of any $k$ elements.
\end{proof}
\end{proposition}

We have now seen that both IP and TP theories have \textbf{c}-TP. In proving that all unstable theories have \textbf{c}-TP, we make use of Shelah's Theorem II.4.7 from \cite{shelah_classification_1990}, which implies that all unstable theories take one of these two forms.

\begin{theorem}{(Shelah's First Dichotomy)}\label{first dichotomy}
    Suppose $T$ is an unstable theory. Then $T$ has IP, SOP, or both.
\end{theorem}

\begin{corollary}{}
If $T$ is unstable, then $T$ has $\textbf{c}$-TP.
\begin{proof}
    By Shelah's First Dichotomy, any unstable theory is either IP or SOP (hence TP); in either case, such a theory has $\textbf{c}$-TP.
\end{proof}
\end{corollary}

In fact, \textbf{c}-TP also implies instability, completing a characterization of this classical notion using the generalized tree property. Just as TP can be taken to be witnessed by an $L_s$-indiscernible tree, one can likewise take an instance of $\calI$-TP to be $L_{s,\kappa,\textbf{c}}$-indiscernible, a fact we will make ready use of in the following proof.

\begin{theorem}\label{c-TP unstable theorem}
    \textbf{c}-TP theories are unstable. Moreover, if $\vphi(x,y)$ is a formula with \textbf{c}-TP, then it is an unstable formula.
\end{theorem}
\begin{proof}
    Suppose $\vphi(x,y)$ is a \textbf{c}-TP formula. Towards a contradiction, suppose $\vphi$ is stable. 
    
    Let $\bbM^{Sk}\models T^{Sk}$ be a Skolemization of the monster model in language $L^{Sk}$. The original witness to \textbf{c}-TP is still a witness after Skolemizing. By compactness, we can take the tree witnessing \textbf{c}-TP to be of height $\kappa>2^{|L^{Sk}|+\aleph_0}$. Since $L_{s,\kappa, \textbf{c}}$-trees have the modeling property, we may take \textbf{c}-TP to be witnessed by an $L_{s,\kappa, \textbf{c}}$-indiscernible tree in $\bbM^{Sk}$.  Say that $A=(a_{\eta})_{\eta\in \textbf{c}^{<\kappa}}$ is such an indiscernible witness.

    Suppose the quantifier-free type $q$ for which siblings are $q$-inconsistent mentions precisely the distinct colors $c_0, c_1,\dotsm, c_{m-1}$, where $m> 1$ (otherwise we have one-color \textbf{c}-TP, which is just ordinary TP, hence $\vphi$ is unstable.) Suppose that $q=c_{i_0}<c_{i_1}<\dotsm <c_{{i_{j-1}}}$, where the $i_0,\dots, i_{j-1}$ may or may not be distinct.
    
    We inductively define a sequence of trees $(\textbf{T}_n)_{n<\omega}$, along with a  sequence of tree elements $(\xi_n)_{n<\omega}$, which satisfy the following:

    \begin{enumerate}[label=(\alph*)]
        \item $\textbf{T}_n$ is a subtree of $\textbf{c}^{<\kappa}$;

        \item $\xi_n\in \textbf{T}_n$;

        \item $\textbf{T}_n$ has $\kappa$-many levels;
        
        \item $\xi_n\trianglelefteq \xi_{n+1}$, and $\xi_n\neq \xi_{n+1}$;
    
        \item The $\xi_n$ rotate through the colors $c_{i_0},\dotsm, c_{i_{j-1}}$, i.e., $\xi_n$ has color $c_{i_\ell}$ if, and only if, $$n\mod j=\ell;$$

        \item If $\xi_{n}\trianglelefteq\eta,\nu\in\textbf{T}_{n+1}$  and $\eta,\nu$ have the same color, then  $$\tp_{L^{Sk}}(a_\eta/\{a_{\xi_i}\colon i\leq n\})=\tp_{L^{Sk}}(a_\nu/\{a_{\xi_i}\colon i\leq n\}).$$
    \end{enumerate}

    Take $\xi_0$ to be an arbitrary element of $\textbf{c}^{<\kappa}$ of color $c_{i_0}$, and let $\textbf{T}_0=\{\eta\in\textbf{c}^{<\kappa}\colon \xi_0\trianglelefteq\eta\}$.
    

    Now, suppose we have chosen $\xi_0,\dotsm, \xi_{n-1}$ and $\textbf{T}_0,\dotsm, \textbf{T}_{n-1}$ satisfying (a-f). 

    Suppose $\xi_{n-1}\triangleleft\eta,\nu\in\textbf{T}_{n-1}$ are two indices of the same color, say $c_t$, at the same level of $\textbf{c}^{<\kappa}$ tree, say level $\alpha$. Then observe that $$\qftp_{L_{s_\textbf{c}}}(\eta,\xi_0,\dotsm, \xi_{n-1})=\qftp_{L_{s_\textbf{c}}}(\nu,\xi_0,\dotsm, \xi_{n-1}).$$
    Thus by the $s_\textbf{c}$ indiscernibility of $A$, we have that 
    $$\tp_{L^{Sk}}(a_\eta/\{a_{\xi_i}\colon i<n\})=\tp_{L^{Sk}}(a_\nu/\{a_{\xi_i}\colon i<n\}).$$
    So there is a single ${L^{Sk}}$-type over $\{a_{\xi_i}\colon i<n\}$ associated to the color $c_t$ and the level $\alpha$, call it $r_{t,\alpha}$. Moreover, there is an $m$-tuple of types associated to all of the colors at level $\alpha$, $$(r_{0,\alpha},r_{1,\alpha},\dotsm, r_{m-1,\alpha}).$$

    By counting, observe that there are at most $2^{|L^{Sk}|+\aleph_0}$-many such $m$-tuples over $\{a_{\xi_i}\colon i<n\}$. Since, by (c), $\textbf{T}_{n-1}$ has height $\kappa$, we may apply the Pigeonhole Principle to find a single $m$-tuple of color types, $(R_{0,n},R_{1,n},\dotsm, R_{m-1,n})$
     and $X_n\subseteq \kappa$ with $|X_n|=\kappa$ such that if $\alpha\in X_n$, then $$(r_{0,\alpha},r_{1,\alpha},\dotsm, r_{m-1,\alpha})=(R_{0,n},R_{1,n},\dotsm, R_{m-1,n}).$$
     Then, let $\textbf{T}_n$ be the restriction of $\textbf{T}_{n-1}$ to extensions of $\xi_{n-1}$ in the levels in $X_n$, i.e, let
     $$\textbf{T}_n=\{\eta\in \textbf{T}_{n-1}\colon \xi_{n-1}\trianglelefteq\eta,\enspace\text{and}\enspace length(\eta)\in X_n,\enspace\text{and if}\enspace length(\xi_{n-1})<\beta \in\dom(\eta)\setminus X_n,\enspace\text{then}\enspace \eta(\beta)=0\}.$$
     Choose $\xi_n\in\textbf{T}_n$ of color $c_{i_\ell}$, where $n\mod j=\ell$.

    It is immediate that $\textbf{T}_n$ and $\xi_n$ satisfy (a) and (b), and (e). (c) holds because $\textbf{T}_n$ is the restriction of a height $\kappa$ tree, $\textbf{T}_{n-1}$, to the levels in $X_n$, a set of size $\kappa$. The inequality condition of (d) holds because $\xi_{n-1}$ and $\xi_n$ have different colors. Finally, (f) holds because of the Pigeonhole argument: two elements of $\textbf{T}_n$ of the same color $c_\ell$ share the type $R_{\ell,n}$ over $\{a_{\xi_i}\colon i<n\}$. So the induction continues.

     Now, consider the partial type $\{\vphi(x, a_{\xi_n})\colon n< \omega\}$. Since (d) implies that it is a subset of the partial type determined by a particular path through $A$, which witnesses \textbf{c}-TP for $\vphi$, we must have that this partial type is consistent. Let $p(x)$ be a complete $\vphi$-type over  $M=Sk(\{a_{\xi_n}\colon n<\omega\})$ extending this partial type. 
    Since we assumed that $\vphi$ was stable, we must have an $M$ formula $d_\vphi(y)$ such that $\vphi(x,c)\in p(x)$ if and only if $\bbM\models d_\vphi(c)$. 

    Since $d_\vphi$ is a formula over $M=Sk(\{a_{\xi_i}\colon i<\omega\})$, there must be some $N<\omega$ such that all of the parameters appearing in $d_\vphi$ come from $Sk(\{a_{\xi_i}\colon i<N\})$.  So there is some $L^{Sk}$-formula $d'_\vphi(y,z_0,\dotsm, z_{N-1})$ over $\emptyset$ such that, for any $c\in|\bbM|$, we have $\bbM\models d_\vphi(c)$ if, and only if, $\bbM^{Sk}\models d'_\vphi(c,a_{\xi_0},\dotsm, a_{\xi_{N-1}})$.
    
    Consider an arbitrary immediate successor $\eta$ of $\xi_{N-1}$; I will argue that we have $\bbM\models d_\vphi (a_\eta)$. Suppose $\eta$ has color $c_\ell$. Consider $\xi_{N'}$ for $N'\geq N$, with $\ell=N'\mod j$, so $\xi_{N'}$ has color $c_\ell$. Since $p'(x)$ extends $\{\vphi(x, a_{\xi_n})\colon n<\omega\}$, which contains $\vphi(x, a_{\xi_{N'}})$, we have $\bbM^{Sk}\models d'_\vphi(a_{\xi_{N'}},a_{\xi_0},\dotsm, a_{i_{N-1}})$. Since $\xi_{N-1}\trianglelefteq \eta, \xi_{N'}\in \textbf{T}_N$, and $\eta$ and $\xi_{N'}$ share color $c_\ell$, (f) implies that $$\tp_{L^{Sk}}(a_{\xi_{N'}}/\{a_{\xi_i}\colon i<N\})=\tp_{L^{Sk}}(a_{\eta}/\{a_{\xi_i}\colon i<N\}).$$
    Hence we may conclude that $\bbM^{Sk}\models d'_\vphi(a_\eta,a_{\xi_0},\dotsm, a_{i_{N-1}})$, and thus $\bbM\models d_\vphi(a_\eta)$. 
    

    We have seen that $\bbM\models d_\vphi(a_\eta)$ for all children of $a_{\xi_{N-1}}$. Hence for all such $a_\eta$, $\vphi(x,a_\eta)$ must be in the unique nonforking extension of $p(x)$ to a type over $\bbM$. Thus $\{\vphi(x, a_\eta)\colon a_\eta\enspace\text{is a child of}\enspace a_{\xi_{N-1}}\}$ is consistent. But these are the children of a common node in $A$, and hence must be inconsistent by the definition of \textbf{c}-TP. This is a contradiction. 
\end{proof}

As an immediate corollary of the above facts, we get the following characterization.

\begin{theorem}
    A formula $\vphi(x,y)$ is unstable if, and only if, it has \textbf{c}-TP.
\end{theorem}

Ultimately, the fact about \textbf{c} used in this proof is that it is Ramsey index structure which has more than one quantifier-free one type, each of which has infinitely many instances. It is therefore possible to extend this result to arbitrary index structures which meet those criteria.

\begin{corollary}\label{generalized stability characterization}
    Suppose $\calI$ is a Ramsey index structure with at least two non-algebraic quantifier-free 1-types. Then a formula $\vphi(x,y)$ is unstable if, and only if, it has $\calI$-TP.
\end{corollary}

As noted in the Introduction, this characterization of the stability property is \textit{positive}, in the sense that the definition of $\calI$-TP uses only $\vphi$-formulas, never $\neg\vphi$.

\section{Colored Tree Properties of the First and Second Kind}

TP is not the only tree property of interest. Thus, knowing that \textbf{c}-TP characterizes stability, a subsequent research question was to understand the colored generalizations of other tree properties. First among these is the Tree Property of the First Kind ($\TPi$):

\begin{definition}{}\label{TP1 definition}
        We say that a formula $\vphi(x,y)$ has $\text{TP}_1$ when there exists a tree of parameters $(a_\eta)_{\eta\in \omega^{<\omega}}$ and natural number $k$ such that:

        \begin{enumerate}
    \item \textbf{Paths are consistent:} For all $\eta\in\omega^{<\omega}$, $\{\vphi(x, a_{\eta\mid i})\colon i\leq ln(\eta)\}$ is consistent.

   \item \textbf{Incomparables are $k$-inconsistent:} For pairwise incomparable $\eta_1,\dotsm, \eta_k\in\omega^{<\omega}$, $\{\vphi(x,a_{\eta_1}),\dotsm, \vphi(x,a_{\eta_k})\}$ is inconsistent.

\end{enumerate}
\end{definition}

In \cite{Kim_Kim_Notions_Around_TP1}, Kim and Kim proved that we may always take $k$ to be equal to 2 in the definition of $\TPi$, and the definition is often presented that way. However, the generalization to \textbf{c}-$\TPi$ is more obvious when stated in the original way, with $k$ arbitrary.

While $\calI$-TP could straightforwardly be defined for any Ramsey index structure $\calI$, the same is not true $\TPi$. The problem is that, in order to make sense of the inconsistency of incomparables together with $q$-inconsistency, relations and functions must be defined \textit{between} arbitrary incomparable elements of $\calI^{<\omega}$. The conception of $\calI^{<\omega}$ as consisting of ``copies of" $\calI$, each of which has as its domain the set of immediate successors of a particular node, allows us to make sense of relations and functions defined on immediate successors. Yet it is less clear how to proceed with incomparables which are not siblings.\footnote{One option, which we hope to explore in future work, involves taking the index structure to be the \textit{free superposition} of the tree structure with a Ramsey index structure $\calI$. As per a result of Bodirsky in \cite{Bodirsky_2012}, this way of amalgamating two Ramsey structures gives a new one. However, it also introduces a new, unwanted ordering, orthogonal to the tree's lexicographic order. For now, we stick to the simple case of colored linear orders, where it is clear how to define the correct $\TPi$ analogue.} However, in a case like that of \textbf{c}, where the expansion of the linear order is only unary, the problem does not arise. Hence, we solely define \textbf{c}-$\TPi$, reserving study of a more general form for future work.

\begin{definition}{}\label{cTP1 definition}
        We say that a formula $\vphi(x,y)$ has \textbf{c}-$\text{TP}_1$ when there exists a tree of parameters $(a_\eta)_{\eta\in \textbf{c}^{<\omega}}$ and a complete quantifier-free $n$-type $q$ in the language of \textbf{c}, such that:

        \begin{enumerate}
    \item \textbf{Paths are consistent:} For all $\eta\in\textbf{c}^{<\omega}$, $\{\vphi(x, a_{\eta\mid i})\colon i\leq ln(\eta)\}$ is consistent.

   \item \textbf{Incomparables are $q$-inconsistent:} For pairwise incomparable $\eta_1,\dotsm, \eta_n\in\textbf{c}^{<\omega}$, if \newline $q= \qftp_{L_{\textbf{c}}}(\eta_1,\dotsm, \eta_n)$, then $\{\vphi(x,a_{\eta_1}),\dotsm, \vphi(x,a_{\eta_n})\}$ is inconsistent.

\end{enumerate}
\end{definition}

We may assume that trees witnessing \textbf{c}-$\TPi$ are $L_{0,\textbf{c}}$-indiscernible: By Proposition \ref{L_0,I trees have modeling property}, such trees have the modeling property, and examining the definition of \textbf{c}-$\TPi$, it is clear that a tree locally based on a witness is still a witness.

The case of TP and \textbf{c}-TP above might suggest that \textbf{c}-$\TPi$ is a strictly weaker notion than $\TPi$, and perhaps carves out a distinct classificatory property of theories. In fact, it recaptures precisely the same notion, even locally (up to a conjunction of the original formula). This shows that the requirement that incomparables give inconsistent formulas is a particularly strong one, leading to a ``collapse" to 2-$\TPi$, both from $k$-$\TPi$ and from \textbf{c}-$\TPi$.

\begin{theorem}{}\label{TP1=cTP1}
    A theory $T$ has \textbf{c}-$\text{TP}_1$ if, and only if, it has $\text{TP}_1$. Moreover, $\vphi(x,y)$ has \textbf{c}-$\TPi$ if, and only if, a conjunction of instances of $\vphi(x,y)$ has $\TPi$.

\begin{proof}
    If $T$ has $\TPi$, then coloring the tree witnessing it however you like will give a witness to \textbf{c}-$\TPi$; this holds locally.
    
    Suppose $T$ has \textbf{c}-$\TPi$, witnessed by $\vphi(x,y)$ and the tree of parameters $A=(a_\eta)_{\eta\in \textbf{c}^{<\omega}}$. Assume that $q$ is the type $c_0<c_1<\dotsm <c_{k-1}$, where $c_0,\dotsm, c_{k-1}$ are color predicates that may or may not be distinct from each other.
    
     We produce a new tree, indexed by $\omega^{<\omega}$, which will witness $\text{TP}_1$ for the formula $\psi(x,\bar{y})=\bigwedge _{0\leq i<k}\vphi(x,y_i)$. This tree will consist of $k$-tuples of successive elements from $A$, cycling through the colors appearing in $q$. The diagram below depicts the tree constructed for $k=2$.
    
    Define a tree of parameters $B=(b_\eta)_{\eta\in\omega^{<\omega}}$ as follows. Let $b_\emptyset=(a_\emptyset,\dotsm, a_\emptyset)$ be a $k$-tuple. We proceed to choose the elements of the tree, making sure that the following are all satisfied for each $\eta\in\omega^{<\omega}$: 

    \begin{enumerate}[label=(\alph*)]
        \item $b_\eta$ is a $k$-tuple $(a_{\eta_0},\dotsm, a_{\eta_{k-1}})$ of (possibly distinct) elements from $A$;

        \item The elements of $A$ appearing in $b_\eta$ and its predecessors lie on some path $\nu$ from $A$. More carefully, let $A_{b_\eta}$ be the set $\bigcup_{\gamma\trianglelefteq\eta}\{a_\xi\colon a_\xi\enspace\text{appears in the tuple}\enspace b_{\gamma}\}$. Then there exists a path $\nu\in\textbf{c}^\omega$ such that $$A_{b_\eta}\subseteq \{a_{\nu(i)}\colon i<\omega\}.$$

        \item If $\eta\neq \emptyset$, then for each $i<k$, $b_\eta$ contains an element of $A$ indexed by a node colored $c_i$.

        \item If $\eta\bot\nu\in\omega^{<\omega}$, then $\eta_i\bot\nu_j$, for $i,j\in\{0,\dotsm, k-1\}$.
    \end{enumerate}

    Suppose we have defined $b_\eta$ satisfying these conditions, for $\eta\in\omega^{n}$. Since, by (a) and (b), $b_\eta$ consists of  tuples from $A$ lying on a path, $A_{b_\eta}$ contains some element with $\trianglelefteq$-maximal index, say $\eta^*\in \textbf{c}^{<\omega}$. Let $(i_k)_{k<\omega}$ be a sequence of distinct elements from \textbf{c} with color $c_0$; assuming \textbf{c} has order type $\omega$, take these to be the even nodes, as depicted in the diagram. Let $d_1,d_2,\dotsm, d_{k-1}\in \textbf{c}$ be elements of colors $c_1,c_2,\dotsm, c_{k-1}$, respectively. 
    Define $$b_{\eta^\smallfrown k}=(a_{\eta^{*\smallfrown}i_k},a_{\eta^{*\smallfrown}i_k^\smallfrown d_1},\dotsm, a_{\eta^{*\smallfrown}i_k^\smallfrown d_1^\smallfrown\dotsm^\smallfrown d_{k-1}}).$$ For this new element of the tree, (a) is satisfied because it is a $k$-tuple, (b) is satisfied because the elements appearing in $b_{\eta^\smallfrown k}$ lie on a path, (c) is satisfied, since the $i_k$ accounts for color $c_0$ and the $d_i$ were chosen to include all the other colors, and (d) is satisfied because the $i_k$s were chosen to be distinct.

    \adjustbox{center, scale=.9}
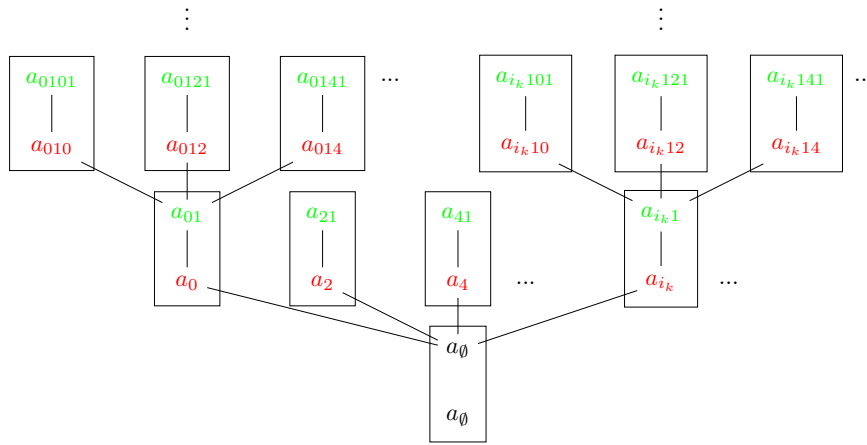
{
\begin{tikzpicture}
    \node(a0r){\color{red}$a_0$};
    \node(a0space1)[right of = a0r]{};
    \node(a1r)[right of = a0space1]{\color{red}$a_2$};
    \node(a0space2)[right of = a1r]{};
    \node(a2r)[right of =a0space2]{\color{red}$a_4$};
    \node(dots1)[right of =a2r]{...};
    \node(a0space3)[right of = dots1]{};
    \node(air)[right of = a0space3]{\color{red}$a_{i_k}$};
    \node (dots2)[right of =air]{...};
    \node(a0g)[above of =a0r]{\color{green}$a_{01}$};\
    \node(a1g)[above of =a1r]{\color{green}$a_{21}$};
    \node(a2g)[above of =a2r]{\color{green}$a_{41}$};
    \node(aig)[above of =air]{\color{green}$a_{i_k1}$};
    \node[draw, fit=(a0r) (a0g)]{};
    \node[draw, fit=(a1r) (a1g)]{};
    \node[draw, fit=(a2r) (a2g)]{};
    \node[draw, fit=(air) (aig)]{};

    \node(a01r)[above of = a0g]{\color{red}$a_{012}$};
    \node(a01g)[above of = a01r]{\color{green}$a_{0121}$};
    \node(a01space1)[left of = a01r]{};
    \node(a00r)[left of = a01space1]{\color{red}$a_{010}$};
    \node(a00g)[above of = a00r]{\color{green}$a_{0101}$};
    \node(a01space2)[right of = a01r]{};
    \node(a02r)[right of = a01space2]{\color{red}$a_{014}$};
    \node(a02g)[above of = a02r]{\color{green}$a_{0141}$};
    \node(dotsa0)[right of = a02g]{...};
    \node[draw, fit=(a01r) (a01g)]{};
    \node[draw, fit=(a00r) (a00g)]{};
    \node[draw, fit=(a02r) (a02g)]{};
    \node(dotsabove0)[above of = a01g]{\vdots};

    \node(ai1r)[above of = aig]{\color{red}$a_{i_k12}$};
    \node(ai1g)[above of = ai1r]{\color{green}$a_{i_k121}$};
    \node(ai1space1)[left of = ai1r]{};
    \node(ai0r)[left of = ai1space1]{\color{red}$a_{i_k10}$};
    \node(ai0g)[above of = ai0r]{\color{green}$a_{i_k101}$};
    \node(ai1space2)[right of = ai1r]{};
    \node(ai2r)[right of = ai1space2]{\color{red}$a_{i_k14}$};
    \node(ai2g)[above of = ai2r]{\color{green}$a_{i_k141}$};
    \node(dotsai)[right of = ai2g]{...};
    \node[draw, fit=(ai1r) (ai1g)]{};
    \node[draw, fit=(ai0r) (ai0g)]{};
    \node[draw, fit=(ai2r) (ai2g)]{};
    \node(dotsabovei)[above of = ai1g]{\vdots};

    \node(empty1)[below of = a2r]{$a_\emptyset$};
    \node(empty2)[below of = empty1]{$a_\emptyset$};
    \node[draw, fit = (empty1) (empty2)]{};

    \draw[black](a0r)--(a0g);
    \draw[black](a1r)--(a1g);
    \draw[black](a2r)--(a2g);
    \draw[black](air)--(aig);
    \draw[black](a0g)--(a01r);
    \draw[black](a0g)--(a00r);
    \draw[black](a0g)--(a02r);
    \draw[black](a00r)--(a00g);
    \draw[black](a01r)--(a01g);
    \draw[black](a02r)--(a02g);
    \draw[black](aig)--(ai1r);
    \draw[black](aig)--(ai0r);
    \draw[black](aig)--(ai2r);
    \draw[black](ai0r)--(ai0g);
    \draw[black](ai1r)--(ai1g);
    \draw[black](ai2r)--(ai2g);
    \draw[black](empty1)--(a0r);
    \draw[black](empty1)--(a1r);
    \draw[black](empty1)--(a2r);
    \draw[black](empty1)--(air);

\end{tikzpicture}
}

    The tree $B$ witnesses $k$-$\TPi$ for $\psi(x,\bar{y})$. By condition (b), the collection of tuples indexed by paths through $B$ consist of elements appearing on paths through the original tree $A$. Since instances of $\vphi$ along those paths are consistent, as a witness to \textbf{c}-$\TPi$, the paths in the new tree will likewise be consistent for $\psi$, as it is a conjunction of instances of $\vphi$. Now, by (c) and (d), each $k$-tuple of incomparable $\eta_1,\dotsm, \eta_{k-1}\in \omega^{<\omega}$ index elements containing incomparables $a_{\nu_0}, a_{\nu_1},\dotsm,a_{\eta_{k-1}}$ such that $(\nu_0,\dotsm, \nu_{k-1})\models q$. Hence the formulas $\{\vphi(x, a_{\nu_0}),\vphi(x, a_{\nu_1}),\dotsm, \vphi(x,a_{\nu_{k-1}})\}$ must be inconsistent, by the $q$-inconsistency of $A$. Hence the formulas $\{\psi(x,b_{\eta_0}),\dotsm, \psi(x,b_{\eta_{k-1}})\}$ must be inconsistent, so we have $k$-$\text{TP}_1$. Note that $\psi(x,\bar{y})$ is a conjunction of instances of $\vphi(x,y)$.

\end{proof}
\end{theorem}

We may also define a modified version of the Tree Property of the Second Kind ($\TPii$). Here is the original definition.

\begin{definition}{}\label{TP2 definition}
        We say that a formula $\vphi(x,y)$ has $\text{TP}_2$ when there exists an array of parameters $(a_{i,j})_{i\in \omega, j\in \omega}$ and a natural number $k$ such that:

        \begin{enumerate}
    \item \textbf{Paths are consistent:} For all $f\colon \omega\to\omega$, $\{\vphi(x, a_{i,f(i)})\colon i\in\omega\}$ is consistent.

   \item \textbf{Rows are $q$-inconsistent:} For every $i\in\omega$ and $j_1,\dotsm, j_k\in \omega$,  $\{\vphi(x,a_{i,j_1}),\dotsm, \vphi(x,a_{i,j_k})\}$ is inconsistent.

\end{enumerate}
\end{definition}

Unlike with $\TPi$, $\TPii$ introduces no issues for offering a generalized definition. 
\begin{definition}{$\calI$-$\text{TP}_2$}
        For an index-model $\calI$, we say that a formula $\vphi(x,y)$ has $\calI$-$\text{TP}_2$ when there exists an array of parameters $(a_{i,j})_{i\in \omega, j\in \calI}$ and a quantifier-free $n$-type $q$ in the language of $\calI$, such that:

        \begin{enumerate}
    \item \textbf{Paths are consistent:} For all $f\colon \omega\to\calI$, $\{\vphi(x, a_{i,f(i)})\colon i\in\omega\}$ is consistent.

   \item \textbf{Rows are $q$-inconsistent:} For every $i\in\omega$ and $j_1,\dotsm, j_n\in \calI$, if $(j_1,\dotsm,j_n)\models q$, then $\{\vphi(x,a_{i,j_1}),\dotsm, \vphi(x,a_{i,j_n})\}$ is inconsistent.

\end{enumerate}
\end{definition}

With this definition, it may immediately be observed that the random graph again provides a distinction between \textbf{c}-$\TPii$ and $\TPii$.

\begin{example}
    The random graph has \textbf{c}-$\TPii$ for the colored linear order.
    \begin{proof}
        The exact same proof for \ref{random graph C-TP} that the random graph has \textbf{c}-TP shows that it has \textbf{c}-$\TPii$, since the parameters chosen form an array.
    \end{proof}
\end{example}

With this fact in hand, we may observe through soft methods that a dichotomy theorem, analogous to Shelah's result that $T$ has TP if, and only if, $T$ has $\TPi$ or $\TPii$, holds in the colored context.

\begin{corollary}{(Colored tree property dichotomy)}
A theory $T$ has \textbf{c}-TP if, and only if, it has \textbf{c}-$\TPi$ or \textbf{c}-$\TPii$.
\begin{proof}
    It is immediate that a witness to \textbf{c}-$\TPi$ or \textbf{c}-$\TPii$ gives an instance of \textbf{c}-TP. To show the converse, suppose $T$ has \textbf{c}-TP but not \textbf{c}-$\TPi$. Since \textbf{c}-$\TPi$  is equivalent to $\TPi$, $T$ does not have $\TPi$, and hence does not have the Strict Order Property (SOP). Since $T$  is \textbf{c}-TP, it is unstable by Theorem \ref{c-TP unstable theorem}. Thus Shelah's First Dichotomy Theorem \ref{first dichotomy} implies that $T$ has SOP or IP, so since $T$ doesn't have SOP, it must have IP. Any theory with IP has a formula encoding the the random graph as an induced subgraph, and the same proof witnessing \textbf{c}-$\TPii$ for the random graph will show that $T$ also has \textbf{c}-$\TPii$. This completes the proof.
\end{proof}
\end{corollary}

We noted that the random graph has IP. It is no coincidence that this structure, typical for simple theories with IP, has \textbf{c}-$\TPii$. For it turns out that \textbf{c}-$\TPii$ is equivalent to IP.

In the next proof, the following notation for the partial type determined by a path will be helpful. 

\begin{notation}
    Suppose that $f\colon \omega\to \calI$ is a path through a set of tuples $(a_\eta)_{\eta\in X}$ indexed by a tree or array $X$. Then let $p_f(x)=\{\vphi(x,a_{f|_i})\colon i\in\omega\}. $
\end{notation}

\begin{theorem}\label{IP=cTP2 theorem}
    $T$ has \textbf{c}-$\TPii$ if, and only if, $T$ has IP. Moreover, $\vphi(x,y)$ has \textbf{c}-$\TPii$ if, and only if, a conjunction of instances of $\vphi(x,y)$ has IP.

\begin{proof}
    It suffices to show that if $T$ has \textbf{c}-$\TPii$, then it has IP, as the other implication follows immediately from the fact that the random graph has \textbf{c}-$\TPii$.

    Suppose that $\vphi(x,y)$ has \textbf{c}-$\TPii$, witnessed by a colored array $A=(a_{i,j})_{i\in\omega, j\in \textbf{c}}$, for a complete quantifier-free type $q$. Note that we can take $A$ to be a strongly \textbf{c}-indiscernible array by Proposition \ref{array indiscernibility}; local basedness will preserve the fact that the array witnesses \textbf{c}-$\TPii$ .

    The proof proceeds by induction on the number of variables in $q$. We assume throughout that $q$ is minimal, in the sense that no complete type $r$ implied by $q$ and in fewer variables gives inconsistent tuples in any of the rows of $A$. Additionally, we may assume that $q$ implies that two of the variables must be satisfied by elements with distinct colors. If $q$ has only one color, then restricting $A$ to elements indexed by that color would give an instance of ordinary $\TPii$, and thus by Proposition 3.5 of \cite{chernikov2014theories}, $T$ would have IP.

    The minimal number of variables appearing in $q$ is two. Consider the case where $q$ implies that the two indices have distinct colors in a particular order, say $R<G$. Suppose $j_0\in \textbf{c}$ is a red index and $j_1\in \textbf{c}$ is a green index such that $j_0<j_1$. Let $g\colon \omega\to \textbf{c}$ be the path where $g(i)=j_0$  if $i$ is even, and $g(i)=j_1$ if $i$ is odd. So this is a path alternating between a particular red element and a particular green element in \textbf{c}. As $g$ is a path through a \textbf{c}-$\TPii$ instance, $p_g(x)=\{\vphi(x,a_{i,g(i)})\colon i<\omega\}$ is consistent. Take a complete $A$-type $p(x)$ extending $p_g(x)$. Consider the formulas $\vphi(x, a_{i,j_1})$ for $i$ even. By our choice of $g$, $\vphi(x,a_{i,j_0})$ is in $p$. But since $(j_0,j_1)\models q$, we must have $\{\vphi(x, a_{i,j_0}), \vphi(x, a_{i,j_1})\}$ inconsistent. So $\neg\vphi(x, a_{i,j_1})\in p(x)$. Thus the set of formulas $\{\vphi(x, a_{i,j_1})\colon i\enspace\text{is odd}\}\cup \{\neg\vphi(x, a_{i,j_1})\colon i\enspace\text{is even}\}$ is a subset of $p$, and hence consistent. But since $A$ is an indiscernible array, $(a_{i, j_1})_{i\in\omega}$ is an indiscernible sequence. This witnesses the fact that $\vphi$ has IP. Note that the same example works if $q$ says just $R,G$ without requiring any particular ordering; this is the only other two variable, two color case.

    Assume for an inductive hypothesis that if $q$ has $m$-many variables, then this instance of \textbf{c}-$\TPii$ gives an instance of IP for $T$. Now suppose that $q$ has $m+1$-many variables. Assume that $q$ implies $c_0<\dotsm <c_{m-2}<R<G$, where the $c_i$ are colors which may or may not be distinct from each other and from $R,G$. While there are cases where the last two colors may be the same, there will always be some two adjacent distinct colors, and by symmetry, we may take these to be the last two. As with the two variable case, the same proof will work for quantifier free types which require fewer order relations.

    Let $j_0<j_1<j_2<\dotsm <j_{m-1}<j_{m}\in \textbf{c}$ be elements of the colors $c_0,c_1,\dotsm, c_{m-2},R,G$, respectively. Define the following paths $f_0,\dotsm, f_{m-2},h$. For $i<m-1$, let $f_i\colon \omega\to \textbf{c}$ be the constant function $f_i(k)=j_i$. Let $h\colon \omega\to \textbf{c}$ be defined by $h(k)=j_{m-1}$  if $k$ is even, and $h(k)=j_{m}$ if $i$ is odd. Like the path $g$ defined above, this path alternates between a red and a green element of \textbf{c}.

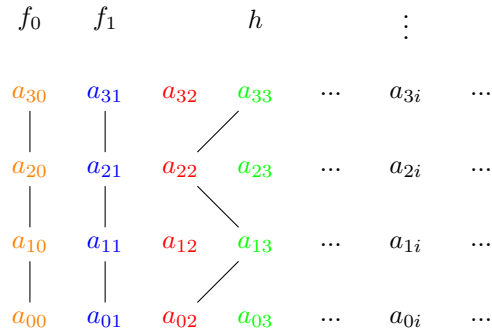
\begin{center}
\begin{tikzpicture}
    \node(a00){\color{orange}$a_{00}$};
    \node(a01)[right of = a00]{\color{blue}$a_{01}$};
    \node(a02)[right of =a01]{\color{red}$a_{02}$};
    \node(a03)[right of =a02]{\color{green}$a_{03}$};
    \node(dots01)[right of =a03]{...};
    \node(a0i)[right of = dots01]{$a_{0i}$};
    \node (dots02)[right of =a0i]{...};

    \node(a10)[above of = a00]{\color{orange}$a_{10}$};
    \node(a11)[right of = a10]{\color{blue}$a_{11}$};
    \node(a12)[right of =a11]{\color{red}$a_{12}$};
    \node(a13)[right of =a12]{\color{green}$a_{13}$};
    \node(dots11)[right of =a13]{...};
    \node(a1i)[right of = dots11]{$a_{1i}$};
    \node (dots12)[right of =a1i]{...};

    \node(a20)[above of = a10]{\color{orange}$a_{20}$};
    \node(a21)[right of = a20]{\color{blue}$a_{21}$};
    \node(a22)[right of =a21]{\color{red}$a_{22}$};
    \node(a23)[right of =a22]{\color{green}$a_{23}$};
    \node(dots21)[right of =a23]{...};
    \node(a2i)[right of = dots21]{$a_{2i}$};
    \node (dots22)[right of =a2i]{...};

    \node(a30)[above of = a20]{\color{orange}$a_{30}$};
    \node(a31)[right of = a30]{\color{blue}$a_{31}$};
    \node(a32)[right of =a31]{\color{red}$a_{32}$};
    \node(a33)[right of =a32]{\color{green}$a_{33}$};
    \node(dots31)[right of =a33]{...};
    \node(a3i)[right of = dots31]{$a_{3i}$};
    \node (dots32)[right of =a3i]{...};

    \node(f0)[above of = a30]{$f_0$};
    \node(f1)[above of = a31]{$f_1$};
    \node(h)[above of = a33]{$h$};
    \node(vdots 1)[above of = a3i]{$\vdots$};

    \draw[black](a00)--(a10);
    \draw[black](a10)--(a20);
    \draw[black](a20)--(a30);
    \draw[black](a01)--(a11);
    \draw[black](a11)--(a21);
    \draw[black](a21)--(a31);
    \draw[black](a02)--(a13);
    \draw[black](a13)--(a22);
    \draw[black](a22)--(a33);
	decoration = {brace}] (empty1) --  (empty2);

\end{tikzpicture}
\end{center}

    Denote by $p_{\bar{f}h}(x)$ the union of the partial types $p_{f_0}(x)\cup p_{f_1}(x)\cup\dotsm\cup p_{f_{m-2}}(x)\cup p_h(x)$. We consider two cases, based on whether $p_{\bar{f}h}(x)$ is consistent.

    \begin{enumerate}
        \item Suppose $p_{\bar{f}h}(x)$ is consistent. Then consider a complete $A$-type $p(x)$ extending it. We proceed as with the base case. Since the array is $q$-inconsistent, our choices of $j_i$ imply that for any $k\in\omega$, the formulas $\{\vphi(x, a_{k,j_0}),
        \vphi(x, a_{k,j_2}),\dotsm, \vphi (x, a_{k, j_{m}})\}$ are inconsistent. For even $k$, $$\{\vphi(x, a_{k,j_0}),
        \vphi(x, a_{k,j_2}),\dotsm, \vphi (x, a_{k, j_{m-1}})\}\subseteq p_{\bar{f}h}(x)\subseteq p(x).$$ Thus, by the completeness of $p$, $\neg \vphi(x,a_{k,j_{m}})\in p(x)$. Moreover, for $k$ odd, $h(k)=j_m$, so $\vphi(x,a_{k,j_m})\in p(x)$. Hence we may conclude that $$\{\vphi(x, a_{k,j_{m}})\colon i\enspace\text{is odd}\}\cup \{\neg\vphi(x, a_{k,j_{m}})\colon i\enspace\text{is even}\}$$ is a subset of $p$, and thus consistent. Since the sequence $(a_{k,j_{m}})_{k\in\omega}$ is indiscernible, as a consequence of array indiscernibility, we have a witness to $\vphi$ having IP.

        \item Suppose $p_{\bar{f}h}(x)$ is inconsistent. Then by compactness there is some even integer $K\in\omega$ such that $p_{f_0|_K}(x)\cup p_{f_1|_K}(x)\cup\dotsm\cup p_{f_{m-2}|_K}(x)\cup p_{h|_K}(x)$ is inconsistent. We then define a new colored array $B=(b_{i,j})_{i\in\omega, j\in\textbf{c}}$ which witnesses \textbf{c}-$\TPii$ for the formula $\psi(x, \bar{y})=\bigwedge_{i<K}\vphi(x,y_i)$. Note that, strictly speaking, \textbf{c} might now be $\textbf{c}'$ with a different number of colors, but we continue to use \textbf{c}; it makes no difference. In particular, this array will have $r$-inconsistency for a quantifier-free type $r$ in $m$-many variables, allowing us to apply the inductive hypothesis. Assume that \textbf{c} has the order type of $\omega$ (so in particular $j\mod m$ is defined), and let $j\in \textbf{c}$ have color $e_\ell$ if and only if $\ell=j\mod m$.

        For each level $i$ of the array, we choose ``blocks" consisting of height $K$ portions of the original array $A$ determined by the paths $f_0,\dotsm, f_{m-1}, h$. To form these blocks, we choose the following values for $b_{i,j}$: For $\ell< m-1$, if $\ell=j\mod m$, then let $$b_{i,j}=(a_{i\cdot K, j_\ell},a_{i\cdot K+1,j_\ell}, a_{i\cdot K+2, j_\ell},\dotsm, a_{i\cdot K+K-1,j_\ell}).$$ If $m-1=j\mod m$, then set take an ``alternating color" block following the path $h$: $$b_{i,j}=(a_{i\cdot K, j_m},a_{i\cdot K+1,j_{m-1}}, a_{i\cdot K+2, j_m},\dotsm, a_{i\cdot K+K-2,j_m}, a_{i\cdot K+K-1, j_{m-1}}).$$

        I claim that $B$ witnesses \textbf{c}-$\TPii$ for $\psi(x,\bar{y})$ with $r=e_0<e_1<\dotsm<e_{m-1}$. First we check that paths are consistent. Observe that the blocks forming the elements of $B$ are ``properly consecutive", i.e, for each $j\in \textbf{c}$, $b_{i,j}$ contains elements of $A$ from only levels $i\cdot K, i\cdot K+1\dotsm, i\cdot K+K-1$. Moreover, elements from each of these levels of $A$ appears only once in $b_{i,j}$. As a consequence, paths in the new array correspond to paths in the original: if $d\colon \omega\to \textbf{c}$ is a path through $B$, there is a path $d'\colon \omega\to \textbf{c}$ such that $$p_{d'}(x)\vdash\{\psi(x, b_{i,d(i)})\colon i\in \omega\}.$$ Since $p_{d'}(x)$ is consistent, as paths are consistent for the original \textbf{c}-$\TPii$ witness, we must have that $\{\psi(x, b_{i,d(i)})\colon i\in \omega\}$ is consistent. 
        
        Now, to check that levels are $r$-inconsistent, we use the strong \textbf{c}-indiscernibility of the original array to reduce to the inconsistency of $p_{f_0|_K}(x)\cup p_{f_1|_K}(x)\cup\dotsm\cup p_{f_{m-2}|_K}(x)\cup p_{h|_K}(x)$. Consider $j_0'<j_1'<\dotsm<j_{m-1}'\in \textbf{c}$, where $j_s'$ has color $e_s$. First, reduce to considering the first row of the array $B$. For $i\in\omega$, the row $(b_{i,j})_{j\in \textbf{c}}$ consists of elements from $K$ consecutive rows in $A$. Since the array $A$ consists of an indiscernible sequence of rows, the type of any $K$ consecutive rows is the same. Hence $\tp((b_{i,j})_{j\in \textbf{c}})=\tp((b_{0,j})_{j\in \textbf{c}})$. So in particular, $\tp(b_{i,j_0'},\dotsm, b_{i,j_{m-1}'})=\tp(b_{0,j_0'},\dotsm, b_{0,j_{m-1}'})$. Finally, for $s<m$, observe that since $j_s'$ has the color, $e_s$, we actually have  $b_{0,j_s'}=(a_{0, j_s},a_{1,j_s},\dotsm, a_{K-1,j_s})$, if $s<m-1$, so $\psi(x,b_{0,j_s'})=\bigwedge_{i<K}\vphi(x,a_{i,j_s}).$ If $s=m-1$, we have $b_{0,j_s'}=(a_{0, j_m},a_{1,j_{m-1}},\dotsm, a_{K-2,j_m},a_{K-1,j_{m-1}})$. All together, we have that
        $\{\psi(x,b_{0,j_0'}),\dotsm, \psi(x, b_{0,j_{m-1}'})\}$ proves the following set of formulas:
        
        \[\Biggl\{ \bigwedge_{i<K}\vphi(x,a_{i,j_0}),\dots, \bigwedge_{i<K}\vphi(x,a_{i,j_{m-2}}), \vphi(x,a_{0,j_m})\land\vphi(x,a_{1,j_{m-1}})\land\dotsm\land\vphi(x,a_{K-2,j_m})\land\vphi(x,a_{K-1,j_{m-1}})\Biggr\}.\]
        
        Observing the definitions of the functions $f_0,\dotsm, f_{m-2},h$, we see that the above is equivalent to $p_{f_0|_K}(x)\cup p_{f_1|_K}(x)\cup\dotsm\cup p_{f_{m-2}|_K}(x)\cup p_{h|_K}(x)$, which was inconsistent by assumption. By the equality of types mentioned above, we conclude that $\{\psi(x,b_{i,j_0'}),\dotsm, \psi(x, b_{i,j_{m-1}'})\}$ is inconsistent, demonstrating the $r$-inconsistency of rows. So we have a witness to \textbf{c}-$\TPii$ with $r$ of length $m$, so by our inductive hypothesis, we have an instance of IP for $T$. 

        Note that this result is semi-local: if $\vphi$ has \textbf{c}-$\TPii$, the proof shows that a conjunction of instances of $\vphi$ has IP.
    \end{enumerate}

\end{proof}
\end{theorem}

\bibliographystyle{alpha}
\bibliography{bibliography}

\end{document}